\documentclass[10pt]{amsart}
\usepackage{latexsym,amsmath,amssymb,graphicx}
\usepackage[all,web]{xy}

\def\urlfont{\DeclareFontFamily{OT1}{cmtt}{\hyphenchar\font='057}
              \normalfont\ttfamily \hyphenpenalty=10000}

\DeclareFontFamily{OT1}{rsfs10}{}
\DeclareFontShape{OT1}{rsfs10}{m}{n}{ <-> rsfs10 }{}
\DeclareMathAlphabet{\mathscript}{OT1}{rsfs10}{m}{n}

\DeclareMathOperator{\im}{Im}       
\DeclareMathOperator{\id}{id}       
\DeclareMathOperator{\Spec}{Spec}   
\DeclareMathOperator{\Ext}{Ext}     
\DeclareMathOperator{\ext}{\mathcal{E}\it{x}\hskip1pt \it{t}\hskip1pt} 
\DeclareMathOperator{\hm}{\mathcal{H}\it{o}\hskip1pt \it{m}\hskip1pt} 
\DeclareMathOperator{\Exc}{Exc}     
\DeclareMathOperator{\Pic}{Pic}     
\DeclareMathOperator{\Sing}{Sing}   
\DeclareMathOperator{\Def}{Def}     

\def \a{\alpha }

\def\so{{S _0}}
\def\P{{\mathbb{P}}}

\def\p2{\mathbb{P}^2}
\def\p3{\mathbb{P}^3}
\def\p4{\mathbb{P}^4}

\def\su{\operatorname{SU}}

\def\so{\operatorname{SO}}

\def\A{\mathbb{A}}
\def\Z{\mathbb{Z}}

\def\C{\mathbb{C}}
\def\R{\mathbb{R}}

\def\Q{\mathbb{Q}}
\def\T{\mathbb{T}}

\theoremstyle{plain}
\newtheorem{theorem}{Theorem}
\newtheorem{corollary}[theorem]{Corollary}
\newtheorem{proposition}[theorem]{Proposition}
\theoremstyle{remark}
\newtheorem{remark}[theorem]{Remark}
\newtheorem{remarks}[theorem]{Remarks}
\newtheorem{conjecture}[theorem]{Conjecture}
\newtheorem{example}[theorem]{Example}
\newtheorem{examples}[theorem]{Examples}
\theoremstyle{definition}
\newtheorem{definition}[theorem]{Definition}

\newcommand{\cy}{Ca\-la\-bi--Yau }
\newcommand{\ka}{K\"{a}hler }

\title[Deforming Geometric Transitions]{Deforming Geometric Transitions}

\author[Michele Rossi]{Michele Rossi}

\dedicatory{Dedicated to Professor D. Gallarati on the occasion of his 90th birthday}

\thanks{This research was partially supported by the MIUR-PRIN ``Geometria delle Variet\`{a} Algebriche'' Research Funds.}

\address{Dipartimento di Matematica, Universit\`a di Torino,
via Carlo Alberto 10, 10123 Torino} \email{michele.rossi@unito.it}

\begin{document}

\begin{abstract}
After a quick review of the wild structure of the complex moduli space of \cy threefolds and the role of geometric transitions in this context (the \cy web) the concept of \emph{deformation equivalence} for geometric transitions is introduced to understand the arrows of the Gross-Reid \cy web as deformation-equivalence classes of geometric transitions. Then the focus will be on some results and suitable examples to understand under which conditions it is possible to get \emph{simple} geometric transitions, which are almost the only well-understood geometric transitions both in mathematics and in physics.
\end{abstract}
\maketitle

\section{Introduction}	

The aim of this paper is that of extending to geometric transitions (see Definition \ref{def:gt}) the well-known concept of \emph{deformation equivalence} of complex manifolds. Geometric transitions have interesting applications both in mathematics, for the study of the wild structure of the moduli space of \cy varieties, and in physics, describing the transition between topologically distinct super-string mo\-dels of \cy vacua. Then deformation equivalence of geometric transitions and \cy 3--folds, seems to give a more direct relation between the string theoretic \cy web and the mathematical Gross-Reid \cy web (see sections \ref{s:string-web} and \ref{s:cy-web}, respectively).

\noindent A large first part of these notes (sections from 2 to 6) has a purely expository purpose. For a broader discussion of these aspects the interested reader is referred to the extensive survey on the subject \cite{R}. A second part (sections 7 and 8) is instead devoted to giving some new ideas, partial results and examples, with the aim of shedding a brighter light on the study of geometric transitions and, more generally, on the \cy moduli space.

\noindent\emph{Deformation equivalence of geometric transitions} is introduced in Definition \ref{def:def-eq_gt}, allowing us, on the one hand, to think of the Gross-Reid \cy web as a kind of quotient of the string theoretic \cy web by means of def-equivalence, and, on the other hand, to isolate a class of geometric transitions (referred to as \emph{simple}, see Definition \ref{def:simple_gt}) having the properties of being well understood both from the physical and the mathematical point of view: in particular \cy threefolds connected by a simple geometric transition turns out to admit the same fundamental group (see Remark \ref{rem:simple}). New results in this context are then given by Proposition \ref{prop:terminal} and Theorem \ref{thm:typeII}. The former gives an easier formulation of def-equivalence between geometric transitions admitting singular loci comprising at most isolated terminal singularities; this is the case e.g. for \emph{small} geometric transitions (see Definition \ref{def:small_gt}). The latter characterizes the \emph{type II} geometric transitions (see Definition \ref{def:typeII_gt}) as \emph{never simple} ones. From the physical point of view, probably the more interesting property of deformation equivalence of both \cy vacua and geometric transitions, is the fact that string theories ``passes through'' such an equivalence, meaning that main parameters of a physical theory must be sought among geometric def-equivalence invariants (see Remark \ref{rem:fisica}).

\noindent This paper ends up by studying when a small geometric transition is actually a simple one: examples of both simple and non-simple small geometric transitions are given. The main result in this context is given by Proposition \ref{prop:small&simple}, giving a necessary cohomological condition for small geometric transitions to be simple.

\section{Calabi--Yau varieties}

\begin{definition}\label{cy-def}
A compact, complex, \ka manifold $Y$ is a \emph{\cy variety} if
\begin{enumerate}
    \item $\bigwedge^{n} \Omega_{Y} =: \mathcal{K}_Y \cong \mathcal{O}_Y$
    \item $h^{p,0}(Y) = 0 \quad \forall 0<p<\dim Y$
\end{enumerate}
A n--dimensional \cy variety will be also called a \emph{\cy
n--fold}.
\end{definition}

\begin{remarks}{\rm
\begin{enumerate}
    \item The given definition of \cy
    variety includes the following lower dimensional cases
    \begin{itemize}
        \item smooth elliptic curves,
        \item smooth $K3$ surfaces.
    \end{itemize}

    \item Observe that although a smooth elliptic curve always admits a projective embedding, this is no longer the case for $K3$ surfaces. By the way, for $\dim Y \geq 3$, the given definition of a \cy variety $Y$ implies that $Y$ is a \emph{projective variety}: the embedding can be fixed by a suitable integer multiple of a rational \ka form near enough to the \ka metric of $Y$.

\end{enumerate}
}
\end{remarks}

\begin{examples}{\rm
\begin{enumerate}
    \item Smooth hypersurfaces of degree $n+1$ in $\P ^n$ (use Adjunction Formula and the Lefschetz Hyperplane Theorem).
    \item Smooth hypersurfaces (if exist!) of a weighted
    projective space $\P (q_0,\ldots,q_n)$ of degree
    $d=\sum_{i=0}^{n}q_i$.
    \item The general element of the anti--canonical system of a
    \emph{sufficiently good} 4--dimensional toric Fano variety (see \cite{Batyrev94}).
    \item Suitable complete intersections.... (iterate the previous
    examples).
    \item The double covering of $\P ^3$ ramified along a smooth
    surface of degree 8 in $\P ^3$ (octic double solid).
\end{enumerate}}
\end{examples}

In dimension greater than or equal to 3 the previous examples give \emph{topologically distinct} complex varieties, implying immediately that \emph{the complex moduli space of \cy $n$--folds, with $n\geq 3$, has to be necessarily disconnected}. This fact apparently clashes with the smaller dimensional cases:
\begin{itemize}
\item the complex moduli space of elliptic curves is given by the modular curve $\Gamma(1)\backslash \mathbb{H}\cong\A^1$ which parameterizes complex structures over the topological torus $S^1 \times S^1$,
\item after Kodaira \cite{Kodaira64}, the complex moduli space of $K3$ surfaces is given by a \emph{smooth, complex, irreducible space of dimension 20}.
\end{itemize}
Anyway if we insist on looking at this situation from the algebraic point of view, then the moduli space of algebraic $K3$ surfaces turns out to be a dramatically more complicated object: the following facts were known to F.~Enriques \cite{Enriques46}:
    \begin{itemize}
        \item $\forall g \geq 3$ there exists a K3 surface of
        degree $2g-2$ in $\P^g$; hence its sectional genus is $g$;
        \item $\forall g \geq 3$ we can obtain a space $\mathcal{M}_g$
        of complex projective moduli of such
        surfaces, by imposing a polarization: $\mathcal{M}_{g}$ is an irreducible, analytic
        variety with $\dim_{\C}\mathcal{M}_{g}=19$;
        \item then the complex moduli space $\mathcal{M}^{alg}$ of algebraic K3
        surfaces is a \emph{reducible} analytic variety and it
        admits a countable number of irreducible components;
        \item there exist K3 surfaces belonging to more than one
        irreducible component of $\mathcal{M}^{alg}$; anyway if we
        restrict to K3's admitting $\Pic \cong \Z$ (they give the
        general element of any irreducible component) then they
        belong to only one irreducible component.
    \end{itemize}
    What could appear to F.~Enriques as a wildly reducible moduli space
    was explained by K.~Kodaira \cite{Kodaira64} as an analytic
    codimension 1 subvariety of a smooth, irreducible, analytic  variety
    $\mathcal{M}$. More precisely:
    \begin{itemize}
        \item there exist analytic non-algebraic K3 surfaces,
        \item the Kuranishi space of any analytic K3 surface is
        smooth and of dimension 20.
    \end{itemize}
    The latter suffices to construct a smooth, irreducible, analytic
    universal family of K3 surfaces: its base $\mathcal{M}$ is the complex
    analytic moduli space of K3 surfaces and $\dim_{\C}\mathcal{M}=20$.
    Moreover $\mathcal{M}^{alg}$ turns out to be a dense subset
    of $\mathcal{M}$.

In other terms, Kodaira recovered an irreducible moduli space for $K3$'s by leaving the algebraic category and working in the bigger category of compact complex surfaces. By this observation, in the late 80's, M.Reid \cite{Reid87} proposed a conjectural construction of a \emph{sort of connected moduli space for \cy 3-folds} suggesting a construction (originally due to F.~Hirzebruch and later called \emph{conifold transition}) to parameterize birational classes of \cy 3-folds by means of moduli of complex structures on suitable non-\ka complex 3-folds given by the connected sum of copies of solid hypertori $S^3\times S^3$: this is the famous \emph{Reid's fantasy}.

\section{Aspects of deformation theory of \cy varieties}

Let $\mathcal{X}\stackrel{f} {\longrightarrow}B$ be a \emph{flat} and \emph{proper},
surjective map of complex spaces such that $B$ is connected and
there exists a special point $o\in B$ whose fibre $X=f^{-1}(o)$ is a, possibly singular, compact complex space.
Then $\mathcal{X}$ is called \emph{a deformation
family of $X$}. If the fibre $X_b=f^{-1}(b)$ is smooth, for some
$b\in B$, then $X_b$ is called \emph{a smoothing of $X$}. Moreover $X_b$ is also called a \emph{deformation} of $X_o$.

\noindent If the morphism $f$ is \emph{smooth} then $\mathcal{X}\stackrel{f} {\longrightarrow}B$ is called a \emph{smooth deformation family}.

\noindent In the following $\T^i_X$ will denote the \emph{ global deformation object of Lichtenbaum--Schlessinger} \cite{L-S}. Since we will always deal with at least \emph{normal} complex algebraic varieties, we can think of $\T^i_X=\Ext ^i \left(\Omega_X,\mathcal{O}_X\right)$, where $\Omega_{X}$ is the sheaf of holomorphic differential forms on $X$. Consider the
\emph{Lichtenbaum--Schlessinger cotangent sheaves} of $X$, $\Theta^i_{X} = \ext
^i\left(\Omega_{X},\mathcal{O}_{X}\right)$. Then $\Theta^0_X = \hm
\left(\Omega_X,\mathcal{O}_X\right)=: \Theta_X$ is the ``tangent"
sheaf of $X$ and $\Theta^i_X$ is supported over $\Sing(X)$, for
any $i>0$. By the \emph{local to global spectral
    sequence} relating the global $\Ext$ and sheaf $\ext$
    (see \cite{Grothendieck57} and \cite{Godement} II, 7.3.3) we get
    \begin{equation*}
    \xymatrix{E^{p,q}_2=H^p\left(X,\Theta_X^q\right)
               \ar@{=>}[r] & \mathbb{T}_X^{p+q}}
    \end{equation*}
giving that
\begin{eqnarray}
    &&\T^0_X \cong H^0(X,\Theta_X)\ , \label{T0} \\
    \label{T-lisci}&&\text{if $X$ is smooth then}\quad \T^i_X\cong H^i(X,\Theta_X)\ ,  \label{T1'}
\end{eqnarray}

Given a deformation family
$\mathcal{X}\stackrel{f}{\longrightarrow}B$ of $X$ \emph{for each
point $b\in B$ there is a well defined linear (and functorial) map
}
\[
    D_b f: \xymatrix{T_b B\ar[r]& \T^1_{X_b}}\quad
    \text{(Generalized Kodaira--Spencer map)}
\]
(see e.g. \cite{Palamodov} Theorem 5.1). Recall that
$\mathcal{X}\stackrel{f}{\longrightarrow}B$ is called
\begin{itemize}
    \item a \emph{versal} (some authors say \emph{complete}) deformation family of $X$ if for any
    deformation family $(\mathcal{Y},X)\stackrel{g}{\longrightarrow}
    (C,0)$ of $X$ there exists a map of pointed complex spaces $h:(U,0)\rightarrow
    (B,o)$, defined on a neighborhood $0\in U\subset C$,
    such that $\mathcal{Y}|_U$ is the \emph{pull--back} of
    $\mathcal{X}$ by $h$ i.e.
    \[
        \xymatrix{&\mathcal{Y}|_U=U\times_B \mathcal{X}\ar[r]\ar[d]^-g&\mathcal{X}\ar[d]^-f\\
                   C& U\ar@{_{(}->}[l]\ar[r]^-h&B}
    \]
    In particular the generalized Kodaira--Spencer map $\kappa(f)$ turns out to be surjective (\cite{Palamodov2}, \S~2.6);
    \item an \emph{effective versal} (or \emph{miniversal}) deformation family of $X$ if
    it is versal and the generalized Kodaira--Spencer map evaluated at $o\in
    B$, $D_of:\xymatrix{T_o B\ar[r]& \T^1_X}$ is injective, hence an isomorphism;
    \item a \emph{universal} family if it is versal and the map $h:U\to B$ is uniquely determined over the neighborhood
    $0\in U\subset C$. This suffices to imply that $f$ is an effective versal deformation of $X$ (\cite{Palamodov2}, \S~2.7.1).
\end{itemize}

\begin{theorem}[Douady--Grauert--Palamodov \cite{Douady74},
\cite{Grauert74}, \cite{Palamodov72} and \cite{Palamodov} Theorems
5.4-6]\label{DGP teorema} Every compact complex space $X$
has an effective versal deformation
$\mathcal{X}\stackrel{f}{\longrightarrow}B$ which is a proper map
and a versal deformation of each of its fibers. Moreover the germ
of analytic space $(B,o)$ is isomorphic to the germ of analytic
space $(q^{-1}(0), 0)$, where $q:\T^1_X\rightarrow \T^2_X$ is a
suitable holomorphic map (the \emph{obstruction map}) such that
$q(0)=0$. In particular if $\T^0_X=0$ then the previous versal effective deformation of $X$ is actually a universal one for all the
fibres close enough to $X$.
\end{theorem}

\begin{definition}[Kuranishi space]\label{Def-definizione}
The germ of analytic space
\[
    \Def(X):=(B,o)
\]
defined in the previous Theorem, is called the
\emph{Kuranishi space of $X$}.
\end{definition}

\begin{theorem}[Bogomolov-Tian-Todorov \cite{Bogomolov78},\cite{Tian87},\cite{Todorov89},\cite{Ran92}]\label{thm:BTT} Any \cy va\-rie\-ty $Y$ have unobstructed deformations, i.e., its Kuranishi space is smooth. In particular this means that $\Def(Y)\cong\T^1_Y$.
\end{theorem}

Since for a \cy variety $Y$ we get
$$\T^0(Y)\cong H^0(Y,\Theta_Y)\cong H^0(Y,\Omega_Y^{n-1}) =0$$
then
\begin{corollary}\label{cor:universale} Every \cy variety $Y$ admits a universal effective family of \cy deformations of $Y$. In particular $h^{n-1,1}(Y)$ turns out to be the dimension of the complex moduli space of $Y$.
\end{corollary}

\subsection{Deformation equivalence of \cy varieties}
In the late 80's R.~Friedman and J.W.~Morgan \cite{Friedman-Morgan} introduced the following equivalence relation between complex manifolds. Here we use notation introduced by F.~Catanese and M.~Manetti in several subsequent discussions of related problems and conjectures \cite{Catanese},\cite{Cat-Waj},\cite{Manetti},\cite{Catanese2008}.

\begin{definition}[Deformation equivalence, \cite{Friedman-Morgan} pg. 10]\label{def:def-eq}  Two complex manifolds $X_1$ and $X_2$ are \emph{direct deformation equivalent} (i.e. \emph{direct def-equivalent}) if there exists a smooth deformation family $\mathcal{X}\stackrel{f}{\longrightarrow} B$ whose base $B$ is an irreducible complex space admitting two points $b_1,b_2\in B$ such that $$X_i=f^{-1}(b_i)\ ,\ i=1,2\ .$$
The equivalence relation generated by direct def-equivalence is called \emph{def-equivalence} (or \emph{deformation type}): this means that two complex manifolds $X$ and $Y$ are def-equivalent (we will write $X\sim Y$) if and only if there exist a positive integer $n$ and smooth manifolds $X_1,\ldots,X_n$ such that
\begin{enumerate}
  \item $X\cong X_1$ and $X_n\cong Y$,
  \item for every $1\leq i\le n-1$, $X_i$ and  $X_{i+1}$ are direct def-equivalent.
\end{enumerate}
\end{definition}

\begin{remark}\label{rem:irriducibilità} Let us observe that:
\begin{enumerate}
  \item F.~Catanese observed that ``in order to analyse deformation equivalence, one may restrict oneself to the case where $\dim(B)= 1$: since two points in a complex space $B$ belong to the same irreducible component of $B$ if and only if they belong to an irreducible curve $B'\subset B$. One may further reduce to the case where $B$ is smooth simply by taking the normalization $B^0\to B_{\text{red}}\to B$ of the reduction $B_{\text{red}}$ of $B$, and taking the pull-back of the family to $B^0$\,'' \cite{Catanese2008}; more generally, up to a resolution of singularities of $B$ and a base change, one can always assume $B$ to be a \emph{smooth and irreducible} complex space; this is also observed by Friedman and Morgan immediately after the definition of def-equivalence: ``Equivalently, deformation type is the equivalence relation generated by de\-cla\-ring that two complex manifolds  are equivalent if they are both fibers in a proper smooth map between two connected complex \emph{manifolds}'' \cite[pg. 10]{Friedman-Morgan};
  \item if a concept of coarse moduli space for the manifolds $X,Y$ is defined, an equivalent formulation of Definition \ref{def:def-eq} is the following: \emph{two complex ma\-ni\-folds $X$ and $Y$ are def-equivalent if and only if they are elements of the same irreducible component of their moduli space.} This is the case e.g. of minimal compact complex surfaces \cite[Def.~23]{CBP}, \cite{Manetti}. For what concerns \cy varieties a coarse moduli space is well defined as a quasi--projective scheme \cite[\S~1.2]{Viehweg}, then such an equivalent definition can be applied.
\end{enumerate}
\end{remark}

\begin{remark}\label{rem:D_Y} Given a \cy manifold $Y$, let $\mathcal{D}_Y$ denote the \emph{def-equivalence class of $Y$}. For what observed in the previous Remark \ref{rem:irriducibilità} and recalling Corollary \ref{cor:universale}, $\mathcal{D}_Y$ can be thought of as $h^{n-1,1}(Y)$--dimensional irreducible complex space giving an irreducible component of the coarse moduli space of \cy manifolds. There are many ways of compactifying such an irreducible component: here we will not discuss this aspect, being beyond the scope of the present paper. Anyway in the following we will assume that a closure $\mathcal{M}$ of a def-equivalence class $\mathcal{D}$ of \cy manifolds will include any singular degeneration of elements in $\mathcal{D}$ carrying either terminal or canonical singularities.
\end{remark}

\begin{remark}\label{rem:intersezione}
The closures $\mathcal{M}_1$ and $\mathcal{M}_2$ of two def-equivalence classes $\mathcal{D}_1$ and $\mathcal{D}_2$, respectively, may admit a common limit point $b\in \mathcal{M}_1\cap\mathcal{M}_2$; M.~Gross exhibited an effective example of this fact \cite{Gross97a}. Recalling that a 3--dimensional terminal singularity is necessarily an isolated singularity, the extension, due to Y.~Namikawa, of the Bogomolov-Tian-Todorov Theorem \ref{thm:BTT} to the Kuranishi space of an isolated 3-dimensional terminal singularity (\cite{Namikawa94}, Theorem 1), allows us to conclude that,
\begin{itemize}
\item
\emph{for \cy 3--folds, given a common point $b\in \mathcal{M}_1\cap\mathcal{M}_2$ then $Y_b$ can't admit terminal singularities}, since $\Def(Y_b)$ is clearly reducible, hence singular.
\end{itemize}
In fact the Gross' example in \cite{Gross97a} exhibits the case of a \cy 3--fold admitting canonical singularities.
\end{remark}

\begin{remark}
By the classical theorem of Ehresmann:
\begin{itemize}
\item\emph{two def-equivalent complex manifolds are orientedly diffeomorphic}.
\end{itemize}
Freidman and Morgan conjectured that the converse could be true: this is the so called \emph{$\text{def}=\text{diff}$ problem}.  Friedman actually proved this equivalence for \emph{compact complex surfaces with $b_1=0$ and Kodaira dimension less than or equal to 1} \cite{Friedman97}: in particular \emph{it holds for $K3$ surfaces}. Notice that

\begin{itemize}
\item \emph{since the $\text{def}=\text{diff}$ problem admits a positive answer for elliptic curves (obvious) and $K3$ surfaces (Friedman \cite{Friedman97}) it makes sense to ask if does it hold for any \cy varieties.}
\end{itemize}

This problem has been negatively settled in dimension $n\geq3$ by Y.~Ruan \cite{Ruan} who showed that the two \cy varieties constructed by M.~Gross in \cite{Gross97a}, and belonging to different irreducible components of the moduli space, are actually diffeomorphic but not symplectomorphic, hence not def-equivalent\footnote{I thank the unknown referee who pointed me out such a Ruan's counterexample of which I was not aware}.

Subsequently, the general problem has been negatively settled by several counter\-examples, the first of which was given by M.~Manetti, then followed by many others by F.~Catanese, Kharlamov-Kulikov, Bauer-Catanese-Grunewald, Catanese-Wajnryb ... (see \cite{CBP} and therein references).
\end{remark}

\section{Geometric transitions}

\begin{definition}\label{def:gt}
Let $Y$ be a \cy $n$--fold and $\phi : Y\rightarrow \overline{Y}$ be
a \emph{birational contraction} onto a \emph{normal} variety. If
there exists a complex deformation (\emph{smoothing}) of
$\overline{Y}$ to a \cy $n$--fold $\widetilde{Y}$, then the process
of going from $Y$ to $\widetilde{Y}$ is called a \emph{geometric
transition} (for short \emph{g.t.}) and denoted by
$T(Y,\overline{Y},\widetilde{Y})$ or by the diagram
\begin{equation*}
    \xymatrix@1{Y\ar@/_1pc/ @{.>}[rr]_T\ar[r]^{\phi}&
                \overline{Y}\ar@{<~>}[r]&\widetilde{Y}}\ .
\end{equation*}
A g.t. $T(Y,\overline{Y},\widetilde{Y})$ is called
\emph{trivial} if $\widetilde{Y}$ is a deformation of $Y$.

\noindent A g.t. $T(Y,\overline{Y},\widetilde{Y})$ is called a
\emph{conifold transition} (for short \emph{c.t.}) if $\overline{Y}$ admits only \emph{ordinary
double points} (nodes) as singularities.
\end{definition}

\begin{remarks}{\rm
\begin{enumerate}
    \item \emph{Trivial geometric transitions may occur}: in fact it is not possible to
    realize non--trivial transitions in dimension less than or equal to 1. For a trivial g.t. in dimension 3 one may e.g. consider Example
4.6 in \cite{Wilson92} where $\phi$ admits an elliptic scroll as
exceptional divisor and contracts it down to an elliptic curve
$C$.
    \item The transition process was firstly (locally) observed by
    H.~Clemens in the study of double solids $V$ admitting at worst
    nodal singularities \cite{Clemens83}: in his Lemma 1.11 he
    pointed out ``the relation of the resolution of the
    singularities of $V$ to the standard $S^3 \times D_3$ to $S^2 \times
    D_4$ surgery".
    \item Let $T(Y,\overline{Y},\widetilde{Y})$ be a geometric transition of \cy 3--folds. Then $\overline{Y}$ can be supposed to carry canonical singularities, at worst (see \cite{Wilson92} and references therein): it is then a limit point of the def-equivalence class $\mathcal{D}_{\widetilde{Y}}$ of $\widetilde{Y}$.
\end{enumerate}}
\end{remarks}

\subsection{The basic example: the conifold transition in $\P ^4$}\label{l'esempio}

The following example, given in \cite{GMS95}, shows that
\emph{non--trivial (conifold) transitions occur when $\dim Y \geq
3$}.

Let $\overline{Y}\subset\P ^4$ be the singular hypersurface given
by the following equation
\begin{equation}\label{equazione}
    x_3 g(x_0,\ldots ,x_4) + x_4 h(x_0,\ldots ,x_4) = 0
\end{equation}
where $g$ and $h$ are generic homogeneous polynomials of degree 4.
$\overline{Y}$ is then the \emph{generic quintic 3--fold
containing the plane $\pi : x_3 = x_4 = 0$}. Then the singular
locus of $\overline{Y}$ is given by
\begin{equation}\label{singolarita}
    \Sing (\overline{Y}) = \{ [x]\in \P ^4 | x_3=x_4=g(x)=h(x)=0\} \ .
\end{equation}
One can then easily prove that:
\begin{itemize}
  \item $\Sing (\overline{Y})$ is composed by 16 nodes,
  \item \emph{(the resolution $Y$)}: $\Sing (\overline{Y})$ can be simultaneously resolved and the
resolution $\phi : Y\rightarrow \overline{Y}$ is a \emph{small
blow up} such that $Y$ is a smooth \cy 3--fold,
  \item \emph{(the smoothing $\widetilde{Y}$)}: $\overline{Y}$ admits the obvious smoothing given by the generic
quintic 3--fold $\widetilde{Y}\subset \P ^4$. In particular
$\widetilde{Y}$ cannot be a deformation of $Y$ i.e. the conifold
transition $T(Y,\overline{Y},\widetilde{Y})$ is not trivial.
\end{itemize}
The latter fact can be easily shown by applying the Lefschetz Hyperplane Theorem and the K\"{u}nneth
Formula to get the following relations on the second Betti numbers:
\begin{eqnarray}\label{betti_nmb}
\nonumber
  b_2 (\widetilde{Y})&=& b_2 (\P ^4) = 1 \\
  b_2 (Y) &=& b_2 (\P ^4 \times \P ^1) = 2
\end{eqnarray}
Therefore $\widetilde{Y}$ and $Y$ cannot be smooth fibers of the
same analytic family.

\subsection{Local topology of a conifold transition}
From now on we will restrict to consider the case $n=3$ of \cy 3-folds.
Then we can observe the following facts (for full details the interested reader is referred to \cite{GR02}, \S 1.1).
\begin{itemize}
  \item[1.] Locally a 3-dimensional node can be described by the local equation
  $$ \overline{U}:=\{z_1z_3 + z_2z_4 = 0\}\subset\C^4\ .$$
  Topologically $\overline{U}$ turns out to be a cone over $S^3 \times S^2$.
  \item[2.] A local resolution of $\overline{U}$ is described by
  $$ \widehat{U}:=\left\{\begin{array}{cc}
                           y_0 z_4 - y_1 z_3 & =\ 0 \\
                           y_0 z_1 + y_1 z_2 & =\ 0
                         \end{array}
  \right\}\subset \C ^4 \times \P ^1\ .$$
  Then there exists a diffeomorphism $\widehat{U}\stackrel{\Phi}{\cong} \R ^4\times S^2$. Moreover $\widehat{U}$ can be identified with the total space of the rank 2
holomorphic vector bundle
$\mathcal{O}_{\P^1}(-1)\oplus\mathcal{O}_{\P^1}(-1)$ over the
exceptional fibre $\P^1_{\C}=\varphi^{-1}(0)$. In particular
$\widehat{U}$ admits a natural complex structure.
  \item[3.] A local smoothing $\widetilde{U}$ of the node $\overline{U}$ can be given by the
1--parameter family $f:\mathcal{U}\rightarrow\R$ where
$$ U_t:=f^{-1}(t) = \{z_1z_3 + z_2z_4 = t\}\subset\C^4\ .$$
Setting $\widetilde{U}:=U_{t_0}$ for some $t_0 \in \R, t_0>0$, then $\widetilde{U}\stackrel{\Psi}{\cong}
S^3\times\R^3$ since it is diffeomorphic to the cotangent bundle $T^* S^3$
of the 3--sphere giving the vanishing cycle of the smoothing. In particular $\widetilde{U}$ admits a natural symplectic structure for which the
vanishing sphere turns out to be a lagrangian submanifold.
\end{itemize}

\begin{theorem}[Clemens \cite{Clemens83} Lemma 1.11, \cite{GR02} Thm. 1.6]\label{clemens lemma}
Let $D_n\subset \R^n$ be the closed unit ball and consider
\begin{itemize}
    \item $S^3\times D_3\subset S^3\times \R^3 \overset{\Psi^{-1}}{\cong} \widetilde{U}$
    \item $D_4\times S^2\subset \R^4\times S^2 \overset{\Phi^{-1}}{\cong} \widehat{U}$
\end{itemize}
Then $\widetilde{D}:= \Psi^{-1}(S^3\times D_3)$ and $\widehat{D}:=
\Phi^{-1}(D_4\times S^2)$ are compact tubular neighborhoods of the
vanishing cycle $\widetilde{S}\subset\widetilde{U}$ and of the
exceptional cycle $\P^1_{\C}\subset\widehat{U}$, respectively.

\noindent Consider the standard diffeomorphism
\begin{equation*}
    \begin{array}{cccc}
      \alpha' : & (\R^4 \setminus \{0\})\times S^2 & \overset{\cong}{\longrightarrow} & S^3 \times (\R^3 \setminus \{0\}) \\
       & (u,v) & \longmapsto &  ( \frac{u}{|u|}, |u|v ) \\
    \end{array}
\end{equation*}
and restrict it to $D_4 \times S^2$. Since
\begin{equation*}
    \partial(D_4 \times S^2)= S^3 \times S^2 = \partial(S^3 \times D_3)
\end{equation*}
observe that $\alpha'|_{\partial(D_4 \times S^2)}= \id |_{S^3
\times S^2}$. Hence $\alpha'$ induces a standard surgery from
$\R^4\times S^2$ to $S^3\times \R^3$.

\noindent Then $\widetilde{U}$ can be obtained from $\widehat{U}$
by removing $\widehat{D}$ and pasting in $\widetilde{D}$, by means
of the diffeomorphism $\alpha:= \Psi^{-1}\circ\alpha'\circ\Phi$.
\end{theorem}

Let us underline a global consequence of the Theorem \ref{clemens lemma}, as a straightforward application of the Seifert-van Kampen Theorem
\begin{corollary}\label{cor:S-vK} A conifold transition does not change the fundamental group.
\end{corollary}

\section{Reid's fantasy}
Since geometric transitions (and in particular \emph{conifold} transitions) may connect topologically distinct \cy 3-folds, M.~Reid thought that they could be the right instrument to recover a sort of \emph{connectedness} of the \cy 3-folds complex moduli space \cite{Reid87}. Quickly, his construction was the following.
\begin{itemize}
  \item[1.] \emph{Assumption}: every projective \cy 3--fold $Y$ is birational to a \cy 3--fold
    $Y'$ such that $H^2(Y')$ is generated by rational
    curves.
  \item[2.] Consequently if $\phi:Y'\rightarrow\overline{Y}$ is the
    morphism contacting all the homologically independent rational curves, then $\overline{Y}$ is always smoothable, by a Friedman result (\cite{Friedman86}, Corollary 4.7), and every smoothing $\widetilde{Y}$ has $b_2(\widetilde{Y})=0$, meaning that $\overline{Y}$ can be smoothed only to \emph{non--\ka} compact
complex 3--folds.
  \item[3.] By results of C.~T.~C.~Wall \cite{Wall66}, any such smoothing $\widetilde{Y}$ has
topological type completely determined by its third Betti number $b_3(\widetilde{Y})$, implying that it is diffeomorphic to a connected sum $\left(
S^3\times S^3\right) ^{\# r}$ of $r$ copies of the \emph{solid
hypertorus} $S^3\times S^3$.
\end{itemize}
Then we get the famous:

\begin{conjecture}[Reid's fantasy]\label{fantasia di Reid}
Up to some kind of inductive limit over $r$, the birational
classes of projective \cy 3--folds can be fitted together, by
means of geometric transitions, into one irreducible family
parameterized by the moduli space $\mathcal{N}$ of complex
structures over suitable connected sum of copies of solid
hypertori.
\end{conjecture}
This conjecture has the further fascinating property of recovering the idea that moduli could be described by studying complex structures over a (hyper)-torus, typical of elliptic curves.

\noindent Unfortunately, the description of the moduli space $\mathcal{N}$ turns out to be a quite hard problem.

\section{The string theoretic \cy web}\label{s:string-web}

\cy 3--folds play a fundamental role in
 10--dimensional string theories: locally, four dimensions give rise
 to the usual Minkowski space--time $M_4$ while the remaining six dimensions (the so called
 \emph{hidden dimensions} for their microscopic extension) are compactified to a geometric model $Y$
 which, essentially to preserve the required supersymmetry, turns out to
 be a \cy 3--fold. Therefore the string theoretic space-time looks like
 \begin{itemize}
                     \item \emph{a locally trivial 10-dimensional bundle whose base is the usual space-time of Einstein and which is locally isomorphic to $M_4\times Y$.  }
                   \end{itemize}
\subsection{The \emph{vacuum degeneracy problem}} In spite of the fact that there are only five
 consistent 10--dimensional super--string theories, actually
 nearly unique via dualities,
 $$
 
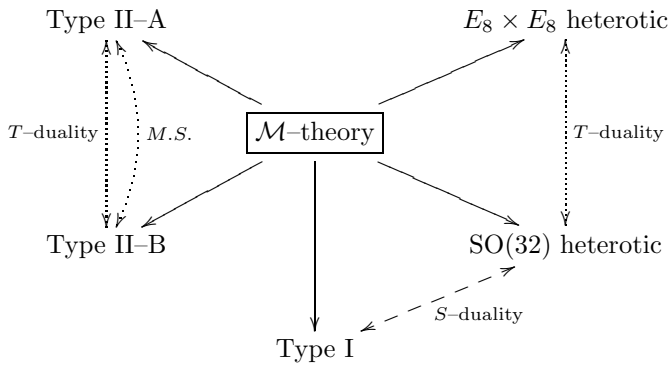
\xymatrix{\text{Type II--A}\ar@{<.>}[dd]_{\text{$T$--duality}}\ar@/^1pc/ @{<.>}[dd]^{M.S.}
                                                                & & E_8\times E_8\ \text{heterotic}
                                                                \ar@{<.>}[dd]^{\text{$T$--duality}} \\
             & \boxed{\text{$\mathcal{M}$--theory}}\ar[ul]\ar[ur]\ar[dl]\ar[dd]\ar[dr]&\\
           \text{Type II--B} & & \so (32)\ \text{heterotic}\ar@{<-->}[dl]^-{\quad\text{$S$--duality}}\\
           & \text{Type I} &}
$$
 the compactification process give
 rise to the problem of choosing the appropriate \cy model: which fibers in the space-time bundle?

 In fact on the physical side, there
 is not any prescription for making a precise choice of the \emph{vacuum model} and on the mathematical side
 there is a multitude of topologically distinct
 \cy 3--folds. By the way, the choice of two distinct models does not at all give
rise to equivalent physical theories, since the physics turns out to be strictly related to the cohomology of the \cy model.

Ideas connected with the formulation of Reid's fantasy,
 suggested to physicists, such as P.~Candelas, P.~S.~Green, T.~H\"{u}bsch
 and others that:
 \begin{itemize}
    \item \emph{(simply connected) \cy 3--folds could be, at least mathematically,
    connected with each other by means of geometric (conifold) transitions}.
 \end{itemize}
This is the so called \emph{\cy web conjecture} described in many
insightful papers starting from 1988. The word "mathematically" in the statement above is a prelude to the further problem of understanding how physics passes through the singularities of a geometric transition process. Actually, as far as I know, conifold transitions are almost the only geometric transitions which have been understood from the physical point of view, after the work of A.~Strominger \cite{Strominger95}: the word ``almost'' refers to the so-called \emph{hyperconifolds transitions}, which are divisorial geometric transitions, discovered by R.~Davies, having the property of being \emph{mirror reverse transitions} of conifold ones (see \cite{Davies10} and \cite{Davies11}).   Let us here underline that the hypothesis of simply connectedness is then necessary if one would only use conifold transitions, as a consequence of Corollary \ref{cor:S-vK}.

\noindent In this sense, conifold transitions turn out to be very interesting both mathematically and physically: probably because they give a concrete bridge between the \emph{complex structures} and the \emph{symplectic structures} on a \cy 3-fold.

\section{The Gross \cy web: nodes and arrows}\label{s:cy-web}
A mathematically refined version of the \cy web conjecture was
presented by M.~Gross in \cite{Gross97b}: it is a sort of synthesis between Reid's fantasy and the \cy web.

The construction.
\begin{itemize}
  \item[1.] On the contrary of the K3 case for which an algebraic K3
surface can be smoothly deformed to a non--algebraic one, the
deformation of a projective \cy 3--fold, even singular, is still
projective.
  \item[2.] Since the hardest part of Reid's fantasy seems to be in dealing with non--\ka 3--folds, one could skip this part by insisting on staying within the projective category.
  \item[3.] Think the \emph{nodes} of the giant web predicted by
the web conjecture as consisting of suitable closures of def-equivalence classes of
\cy 3--folds, as described in Remarks \ref{rem:irriducibilità} and \ref{rem:D_Y} below.
  \item[4.] Two such nodes, say $\mathcal{M}_1$ and
$\mathcal{M}_2$, are connected by an \emph{arrow}
$\mathcal{M}_1\rightarrow\mathcal{M}_2$ if \emph{there exist \cy 3-folds $Y\in\mathcal{M}_1$ and $\widetilde{Y}\in\mathcal{M}_2$ which are each other connected by means of a geometric transition}. More precisely there exists:
\begin{itemize}
  \item a birational contraction to a normal 3--fold
$\phi:Y\rightarrow\overline{Y}$
  \item a deformation family
$(\overline{\mathcal{Y}},\overline{Y})\rightarrow(\Delta,0)$ such that
$\overline{\mathcal{Y}}_t\cong\widetilde{Y}$ for some $t\in\Delta\ ,\ t\neq 0$.
\end{itemize}
\end{itemize}

\begin{example}[See also \cite{Gross97b}]\label{esempio-web}{\rm Let
\begin{itemize}
    \item $\mathcal{M}_Q$ be the family of quintic 3--folds in $\P
    ^4$,
    \item $\mathcal{M}_D$ be the family of double solids (i.e. double
    covers of $\P ^3$) branching along a octic surface of $\P ^3$,
    \item $\mathcal{M}_T$ be (a closure of) the def-class of a smooth blow--up of a
    quintic 3--fold having a triple point.
\end{itemize}
\emph{Then these deformation families are nodes of the
following connected graph}
\begin{equation}\label{3web}
    \xymatrix{&\mathcal{M}_T\ar[dr]\ar[dl]&\\
              \mathcal{M}_Q&&\mathcal{M}_D}
\end{equation}
where the two arrows are obtained as follows:
\begin{itemize}
  \item let $Z$ be a smooth element in $\mathcal{M}_T$ and
$\phi:Z\rightarrow\overline{Y}$ be the contraction of the
exceptional divisor of $Z$. Then $\overline{Y}$ is a quintic
3--fold in $\P ^4$ with a triple point. Since $\overline{Y}$ can
be smoothed to a quintic 3--fold we have
$\mathcal{M}_T\longrightarrow\mathcal{M}_Q$
  \item if we project $\overline{Y}$ from its
triple point then we get a rational morphism $\psi: \overline{Y}\dashrightarrow\P ^3$ which can be lifted to the blow up
$Z$ of $\overline{Y}$, giving rise to a generically finite morphism $\widehat{\psi}:
Z\rightarrow \P ^3$. Consider its Stein factorization
$\widehat{\psi}=f\circ\varphi$. Then we get the following
commutative diagram
\begin{equation}\label{diagrammaStein}
    \xymatrix{Z\ar[r]^{\varphi}\ar[dr]^{\widehat{\psi}}\ar[d]^{\phi}& \overline{X}\ar[d]^{f}\\
              \overline{Y}\ar@{-->}[r]^{\psi}& \P ^3}
\end{equation}
where $f$ gives to $\overline{X}$ the structure of a double
solid branched along a singular octic surface $S\subset\P ^3$.
Since $\overline{X}$ can immediately be smoothed by
smoothing the branching locus $S\subset\P ^3$ this gives the arrow $\mathcal{M}_T\longrightarrow\mathcal{M}_D$\ .
\end{itemize}}
\end{example}

\begin{conjecture}[of Connectedness]\label{connessione}
The graph of (simply connected) \cy 3--folds is connected. Then their moduli can be described by starting from the \emph{primitive} nodes given by def-classes of \cy 3-folds which do not admits any birational contraction landing to a projective normal 3-fold (in general those having Picard number 1).
\end{conjecture}

\emph{A major evidence} for this conjecture is given by Chiang-Greene-Gross-Kanter in \cite{CGGK96} where the authors announced that, by computer procedure, \emph{it is possible to settle in a connected graph all known
examples of \cy hypersurfaces in a 4-dimensional weighted projective space (7555 \cy
3--folds)}. Actually this web is even bigger since many weighted hypersurfaces has been connected passing through hypersurfaces of more general toric, Fano 4-dimensional varieties.

\noindent Moreover, lot of arrows in the previous big connected graph are not generated by conifold transitions but they are represented by very general geometric transitions. In fact many nodes of such a big connected graph are def-classes of non simply connected \cy 3-folds, allowing us to drop the simply connectedness hypothesis in the statement of the Connectedness Conjecture \ref{connessione}.

\section{Deformation of a morphism}
Let $\phi:Y\rightarrow X$ be a morphism of complex spaces and let $B$ be a connected complex space with a special point $o\in B$ such that $g:(\mathcal{Y},Y)\rightarrow(B,o)$ and $f:(\mathcal{X},X)\rightarrow(B,o)$ are deformation families of $Y$ and $X$, respectively. Then a \emph{deformation family of the morphism $\phi$} is a morphism $\Phi:\mathcal{Y}\rightarrow\mathcal{X}$ such that the following diagram commutes
\begin{equation}\label{deformazione di morfismo}
\xymatrix{Y\ar@{^{(}->}[rrr]\ar[rd]^{\phi}\ar[ddr]&&&\mathcal{Y}\ar[dl]_{\Phi}\ar[ddl]^{g}\\
          &X\ar@{^{(}->}[r]\ar[d]&\mathcal{X}\ar[d]_{f}&\\
          &o\ar@{}[r]|{\in}&B&}
\end{equation}
with $Y=\overline{g}^{-1}(o)$, $X=f^{-1}(o)$
    and $\phi = \Phi_{|g^{-1}(o)}$.

    \noindent Given two distinct points $b_1,b_2\in B$ the morphism $\phi_2:=\Phi_{|g^{-1}(b_2)}$ is called a \emph{deformation} of the morphism $\phi_1:=\Phi_{|g^{-1}(b_1)}$ and viceversa.

    \noindent Let us now introduce a non-standard notation: the morphism deformation family $\mathcal{Y}\stackrel{\Phi}{\rightarrow}\mathcal{X}$ is called \emph{smooth} if the deformation family $\mathcal{Y}\stackrel{g}{\rightarrow}B$ is smooth. In this case $\phi_2$ will be called a \emph{smooth} deformation of $\phi_1$ and viceversa.

\subsection{Deformation equivalence of geometric transitions}
\begin{definition}\label{def:def-eq_gt} Two geometric transitions $T_1(Y_1,\overline{Y}_1,\widetilde{Y}_1)$ and $T_2(Y_2,\overline{Y}_2,\widetilde{Y}_2)$ are \emph{direct deformation equivalent} (i.e. \emph{direct def-equivalent}) if
\begin{itemize}
  \item[1.] $Y_1$ and $Y_2$ are both fibers of a same smooth deformation family $\mathcal{Y}\stackrel{f}{\to} B$ over an irreducible base $B$,
  \item[2.] $\widetilde{Y}_1$ and $\widetilde{Y}_2$ are both fibers of a same smooth deformation family $\widetilde{\mathcal{Y}}\stackrel{\widetilde{f}}{\to} \widetilde{B}$ over an irreducible base $\widetilde{B}$,
  \item[3.] $\overline{Y}_1$ and $\overline{Y}_2$ are both fibers of a same deformation family $\overline{\mathcal{Y}}\stackrel{\overline{f}}{\to} \overline{B}$ and, up to shrink $B$, there exist a map $\varphi:B\to\overline{B}$ and a morphism deformation family to the pull-back family $\varphi^*\overline{\mathcal{Y}}=B\times_{\overline{B}}\overline{\mathcal{Y}}$
      \begin{equation*}
    \xymatrix{\mathcal{Y}\ar[rr]^-{\Phi}\ar[dr]_{f}&&\varphi^*\overline{\mathcal{Y}}\ar[dl]^-{\varphi^*\overline{f}}\\
    &B&}
\end{equation*}
such that $\Phi_{|Y_i}=\phi_i$, for $i=1,2$. In particular the birational contractions $\phi_1:Y_1\rightarrow \overline{Y}_1$ and $\phi_2:Y_2\rightarrow \overline{Y}_2$ are smooth deformations of each other. With a slight abuse of notation, in the following we will denote the pull-back family $\varphi^*\overline{\mathcal{Y}}$ by $\overline{\mathcal{Y}}\stackrel{\overline{f}}{\to} B$.
\end{itemize}
The equivalence relation of geometric transitions generated by direct def-equivalence is called \emph{def-equivalence} (or \emph{deformation type}) of geometric transitions.
We will write $T_1\sim T_2$ for def-equivalent geometric transitions.
\end{definition}

Let us observe that the statement 4 at the beginning of section \ref{s:cy-web} defines what means that two nodes are connected by an arrow, but it does not give a concrete definition of what an arrow is. Actually an arrow is defined by a geometric transition connecting smooth elements of two nodes, meaning that an arrow and its defining geometric transition have to be thought of as the same object. On the other hand, notice that \emph{def-equivalent transitions} connect the same def-equivalence classes of \cy 3-folds i.e. the same nodes of the Gross \cy web. It seems then natural to redefine an arrow as follows:
\begin{definition}
  An arrow of the Gross \cy web is a def-equivalence class of geometric transitions.
\end{definition}
As a consequence we get that
\begin{itemize}
  \item \emph{the Gross \cy web is a sort of quotient, up to def equivalence, of the string theoretic \cy web restricted to the algebraic category.}
\end{itemize}

The previous Definition \ref{def:def-eq_gt} can be simplified by putting hypothesis on the singular loci of def-equivalent geometric transitions:

\begin{proposition}\label{prop:terminal} Let $T_1(Y_1,\overline{Y}_1,\widetilde{Y}_1)$ and $T_2(Y_2,\overline{Y}_2,\widetilde{Y}_2)$ be geometric transitions such that both $\Sing(\overline{Y}_1)$ and $\Sing(\overline{Y}_2)$ comprise at most isolated terminal singularities. Then $T_1$ and $T_2$ are direct def-equivalent if and only if their associated birational contractions $\phi_i:Y_i\rightarrow\overline{Y}_i$ are smooth deformations of each other.
\end{proposition}

\begin{proof} Assume that $T_1$ and  $T_2$ are direct def-equivalent. Then there exist two commutative diagrams
\begin{equation*}
\xymatrix{Y_i\ar@{^{(}->}[rrr]\ar[rd]^{\phi}\ar[ddr]&&&\mathcal{Y}\ar[dl]_{\Phi}\ar[ddl]^{f}\\
          &\overline{Y}_i\ar@{^{(}->}[r]\ar[d]&\overline{\mathcal{Y}}\ar[d]_{\overline{f}}&\\
          &b_i\ar@{}[r]|{\in}&B&}\ \quad (i=1,2)\ .
\end{equation*}
meaning that $\phi_1$ and $\phi_2$ are deformations of each other. Moreover $Y_1$ and $Y_2$ are direct def-equivalent \cy 3-folds, meaning that the morphism deformation family $\Phi$ has to be smooth.

\noindent Viceversa if $\phi_1$ and $\phi_2$ are smooth deformations of each other then the deformation family $\mathcal{Y}\stackrel{f}{\rightarrow}B$ in the previous diagrams is necessarily smooth, giving the direct def-equivalence of $Y_1$ and $ Y_2$. To prove that $\widetilde{Y}_1$ and $\widetilde{Y}_2$ are direct def-equivalent \cy 3--folds observe that, under the notation of Definition \ref{def:def-eq_gt}, the generic fibre of the pull-back family $\left\{\overline{\mathcal{Y}}\stackrel{\overline{f}}{\rightarrow}B\right\}=\left\{\varphi^*\overline{\mathcal{Y}}\stackrel{\varphi^*\overline{f}}
{\longrightarrow}B\right\}$ admits at most isolated terminal singularities. Recalling Remark \ref{rem:intersezione}, we are able to assume that the image $\varphi(B)$ may live over an irreducible component of $\overline{B}$ belonging to the closure of a unique def-equivalence class of \cy threefolds. Since $\widetilde{Y}_i$ is a smoothing of $\overline{Y}_i$ this suffices to prove that both $\widetilde{Y}_1$ and $\widetilde{Y}_2$ have to belong to the same def-equivalence class.
\end{proof}

Since def-equivalence is the equivalence relation generated by direct def-e\-qui\-va\-len\-ce, the previous Proposition \ref{prop:terminal} gives immediately the following

\begin{corollary}\label{cor:terminal} Let $T_1(Y_1,\overline{Y}_1,\widetilde{Y}_1)$ and $T_2(Y_2,\overline{Y}_2,\widetilde{Y}_2)$ be geometric transitions such that both $\Sing(\overline{Y}_1)$ and $\Sing(\overline{Y}_2)$ comprise at most isolated terminal singularities. Then $T_1\sim T_2$ if and only if their associated birational contractions $\phi_i:Y_i\rightarrow\overline{Y}_i$ are connected by a finite chain of smooth morphism deformations.
\end{corollary}

Let us say that hypothesis of the previous Proposition \ref{prop:terminal} and Corollary \ref{cor:terminal} are satisfied e.g. by \emph{small} geometric transitions (see the following Definition \ref{def:small_gt}).

\begin{remark}\label{rem:fisica}{\rm  \emph{The physics ``passes through'' def equivalence}, in the sense that it is \emph{def-equivariant}.

\noindent In fact a conformal field theory on a \cy 3-fold $Y$ is the datum of a point $\chi$ on the \emph{complexified \ka cone}
$$\mathcal{K}_{\C}(Y):=\left.\left\{\chi\in H^2(Y,\C)\ |\ \Im(\chi)\in\mathcal{K}(Y)\right\}\right/H^2(Y,\Z)$$
where $\mathcal{K}(Y)$ is the \ka cone of $Y$ and the action of integral cohomology is additive on the real part $\Re(\chi)$. In fact the imaginary part $\omega:=\Im(\chi)$ gives the \ka, Ricci flat metric of $Y$, which is the \emph{geo-metric properties of the supersymmetric string theoretic vacuum}, while the real part $b:=\Re(\chi)$ gives the so called \emph{b-field} describing \emph{the strings' charge properties} of the theory.

\noindent If $Y_1\sim Y_2$ are two def equivalent \cy 3-folds, then there exists an orientation preserving diffeomorphism $f:Y_1\cong Y_2$. In general $f$ induces a contravariant isomorphism $f^*:\mathcal{K}_{\C}(Y_2)\cong \mathcal{K}_{\C}(Y_1)$ (\cite{Wilson92}, Main Thm.). Then the physical theories $(Y_2,\chi)$ and $(Y_1, f^*(\chi))$ are isomorphic. In a special case, it can happen that we are dealing with a couple $(Y,\chi)$ where $Y$ contains a conic bundle over an elliptic curve: in this case, if $Y'$ is a general smooth deformation of $Y$ then the associated orientation preserving diffeomorphism $f:Y'\cong Y$ gives rise to a strict inclusion $f^*:\mathcal{K}_{\C}(Y)\hookrightarrow \mathcal{K}_{\C}(Y')$ (see \cite{Wilson97}) meaning that one can always reduce the physical theory $(Y,\chi)$ to the isomorphic general theory $(Y',f^*(\chi))$.}
\end{remark}

\section{Simple geometric transitions}

Since both in mathematics and in physics the conifold transitions are the most understood geometric transitions between \cy 3--folds, it makes sense to ask \emph{when a geometric transition is def-equivalent to a conifold one.}
Let us then set the following

\begin{definition}[Simple geometric transitions and arrows]\label{def:simple_gt} A g.t. is called \emph{simple} if it is def-equivalent to a conifold transition. Therefore an arrow is called simple if it is the def-equivalence class of a conifold transition.
\end{definition}

\begin{remark}[The importance of being simple]\label{rem:simple}{\rm If $T(Y,\overline{Y},\widetilde{Y})$ is a simple g.t. then, for what observed above:
\begin{itemize}
  \item it is physically well understood by the Remark \ref{rem:fisica},
  \item there exist finite open coverings $\{U_i\}_{i=1}^N$ and $\{\widetilde{U}_i\}_{i=1}^N$ of $Y$ and $\widetilde{Y}$, respectively, and almost everywhere defined diffeomorphisms from $U_i$ to $\widetilde{U}_i$  $$\a_i:U_i\backslash\Exc(\phi)\stackrel{\cong}{\longrightarrow}\widetilde{U}_i\backslash\mathcal{V}\ ,\ 1\leq i\leq N\ ,$$
      where $\mathcal{V}\subset\widetilde{Y}$ is the vanishing locus. In particular $Y$ and $\widetilde{Y}$ have the same fundamental group.
  \end{itemize}
  In fact if $T'(X,\overline{X},\widetilde{X})$ is a conifold t. with $T\sim T'$, then there are orientation preserving diffeomorphisms $Y\cong X$ and $\widetilde{Y}\cong \widetilde{X}$. Then end up by applying Theorem \ref{clemens lemma} and Corollary \ref{cor:S-vK}. In particular the cardinality $N$ of the open coverings is that of the singular locus of $\overline{X}$, i.e. $N=|\Sing(\overline{X})|$.}
\end{remark}

\subsection{Type II geometric transitions are not simple}

\begin{definition}\label{def:typeII_gt} A \emph{type II geometric transition} $T(Y,\overline{Y},\widetilde{Y})$ is a g.t. such that
\begin{itemize}
  \item the associated birational morphism $\phi:Y\rightarrow\overline{Y}$ is \emph{primitive}, i.e. it cannot be factored into birational morphisms of normal varieties,
  \item $\phi$ contracts a divisor down to a point; in
    this case the exceptional divisor $E$ is irreducible and in particular it is a
    (generalized) \emph{del Pezzo surface} (see \cite{Reid80}).
\end{itemize}
\end{definition}

\begin{example}{\rm The g.t. $T(Z,\overline{Y},\widetilde{Y})$ representing the arrow $\mathcal{M}_T\longrightarrow\mathcal{M}_Q$ in the Example \ref{esempio-web}, is a type II g.t.}
\end{example}

By exhibiting a suitable weighted blow down, one can easily produce a type II g.t. $T(Y,\overline{Y},\widetilde{Y})$ such that $Y$ and $\widetilde{Y}$ do not admit the same fundamental group. Then $T$ \emph{cannot be a simple g.t.} due to the previous Remark \ref{rem:simple} and Corollary \ref{cor:S-vK}. Actually a much stronger result can be established:

\begin{theorem}\label{thm:typeII} A type II g.t. is never simple.
\end{theorem}

\noindent\textbf{Proof.} Let us first of all show that a type II g.t. $T_1(Y_1,\overline{Y}_1,\widetilde{Y}_1)$ cannot be direct def-equivalent to a conifold $T_2(Y_2,\overline{Y}_2,\widetilde{Y}_2)$. On the contrary, let us assume the existence of a deformation family of morphisms
\begin{equation*}
    \xymatrix{\mathcal{Y}\ar[rr]^-{\Phi}\ar[dr]_-f&&\overline{\mathcal{Y}}\ar[dl]^-{\overline{f}}\\
    &B&}
\end{equation*}
over an irreducible base $B$ and such that $\mathcal{Y}\stackrel{f}{\rightarrow}B$ is a smooth family realizing the direct def-equivalence of $Y_1$ and $Y_2$. In particular there exist two distinct points $b_1,b_2\in B$ such that $\Phi_{b_i}:=\Phi_{|f^{-1}(b_i)}=\phi_i:Y_i\rightarrow\overline{Y}_i$. Since $T_2$ is conifold then $\Exc(\phi_2)$ is composed by a finite number of disjoint smooth rational curves whose normal bundle is given by $\mathcal{O}_{\P^1}(-1)\oplus\mathcal{O}_{\P^1}(-1)$ (so called \emph{$(-1,-1)$-curves}). Then any $(-1,-1)$-curve is a \emph{stable submanifold} of $Y_2$ in the sense of Kodaira \cite{Kodaira63}. For this reason, up to shrink the irreducible base $B$, we may now assume $B$ to be a suitable neighborhood $B^{\circ}$ of $b_2$ with the addition of a closure point given by $b_1$, such that $\Phi_{b}:Y_b\rightarrow\overline{Y}_b$ turns out to be a c.t. for any $b\in B$ with $b\neq b_1$. The contradiction is then reached by observing that $\overline{Y}_1$ is $\Q$-factorial while $\overline{Y}_b$ can never be $\Q$-factorial for any $b\neq b_1$: this fact is against a result of J.~Koll\'{a}r and S.~Mori (\cite{Kollar-Mori92} Thm. (12.1.10)) guaranteeing that $\Q$-factoriality of the fibers has to be an \emph{open condition} for the pull-back deformation family $\overline{f}:\overline{\mathcal{Y}}\rightarrow B$ i.e. that there should exist an open neighborhood of $b_1\in B$ over which all fibers should be $\Q$-factorial.

\noindent Let us now assume that $T_1\sim T_2$, meaning that there exist a finite sequence of smooth morphism deformation families connecting the birational contractions $\phi_1$ and $\phi_2$. Starting from the last family, the previous argument shows that this family cannot admit a type II birational contraction as a morphism fiber. In particular this holds for the common morphisms to the last and the penultimate families. Hence the same arguments shows that the penultimate family cannot admit a type II birational contraction as a morphism fiber, and so on until we land at the first family giving an absurd.

\subsection{An example of a non-simple small geometric transition}

\begin{definition}[Small g.t.]\label{def:small_gt} A g.t. $T(Y,\overline{Y},\widetilde{Y})$ is called \emph{small} if the associated birational morphism $\phi:Y\rightarrow \overline{Y}$ is a \emph{small birational contraction}, i.e. its exceptional locus $\Exc(\phi)$ has codimension greater than 1 in $Y$.
\end{definition}

Possible exceptional and singular loci occurring in a small g.t. are completely classified (see \cite{R1}, Thm. 6 and references therein):
\begin{itemize}
  \item $\Sing(\overline{Y})$ turns out to be composed by a finite number of isolated \emph{compound Du Val (cDV)} singular points, which in particular are terminal singularities,
  \item $\Exc(\phi)$ is then composed by a finite number of trees of transversally intersecting rational curves, dually represented by ADE Dynkin graphs.
\end{itemize}

Due to the particular geometry of the exceptional locus $\Exc(\phi)$ it is quite natural to ask for the \emph{simplicity of any small geometric transitions}. Unfortunately this is not the case, as the following example shows.

\begin{example}\label{ex:nami}{\rm
The following example is essentially due to Y.~Namikawa (\cite{Namikawa02}, Example 1.11).

Let $S$ be the rational elliptic surface with
sections obtained as the Weierstrass fibration associated with the
bundles homomorphism
\begin{eqnarray}\label{W-Nami}
    (0,B)&:&\xymatrix@1{\mathcal{E}=\mathcal{O}_{\P^1}(3)\oplus\mathcal{O}_{\P^1}(2)
\oplus\mathcal{O}_{\P^1}\ \ar[rr]&&\quad
    \mathcal{O}_{\P^1}(6)}\\
\nonumber
    &&\xymatrix@1{\hskip1cm (x,y,z)\hskip1cm\ar@{|->}[rr]&&\quad - x^2z + y^3 + B(\lambda)\
    z^3}
\end{eqnarray}
for a generic $B\in H^0(\P^1,\mathcal{O}_{\P^1}(6))$ i.e. $S$ is
the zero locus of $(0,B)$ in the projectivized bundle
$\P(\mathcal{E})$. Then:
\begin{itemize}
    \item[1.] \emph{the natural fibration $S\rightarrow\P^1$ has generic smooth fibre
and 6 distinct cuspidal fibres},
    \item[2.] \emph{the fiber product $X:= S\times_{\P^1}S$ is a threefold admitting 6
    singularities of type $II\times II$, in the standard Kodaira
    notation} \cite{Kodaira64},
    \item[3.] \emph{$X$ admits a \emph{small} resolution $\widehat{X}\stackrel{\phi}{\longrightarrow} X$
    whose exceptional locus is composed by 6 disjoint couples of rational curves
    intersecting in one point i.e. 6 disjoint $A_2$ exceptional
    trees,}
    \item[4.] \emph{by results of C.~Schoen, $X$ is a special fibre of the family of fiber products $S_1\times_{\P^1} S_2$
    of rational elliptic surfaces with sections: in particular for $S_1$ and $S_2$
    sufficiently general $\widetilde{X}=S_1\times_{\P^1}S_2$ is a \cy
    threefold giving a smoothing of $X$} (\cite{Schoen88} \S 2).
\end{itemize}
Since $\phi$ is a small, crepant resolution, $\widehat{X}$ turns
out to be a \cy threefold and $T(\widehat{X},X,\widetilde{X})$
is a small \emph{non--conifold} g.t.. Let $p$ be one of the six
singular points of $X$, locally defined as a germ of singularity
by the polynomial
$$
F:=x^2-z^2-y^3+w^3\in \C[x,y,z,w]\ .
$$
Consider the \emph{localization near
to $p$}
\begin{equation}\label{localizzazione}
    \xymatrix{\widehat{U}_p:=\phi^{-1}(U_p)\ar@{^{(}->}[r]\ar[d]^-{\varphi}
    & \widehat{X}\ar[d]^-{\phi}\\
    U_p:=\Spec\mathcal{O}_{F,p}\ar@{^{(}->}[r]&X}
\end{equation}
which induces, since $p$ is a rational singularity, the following
commutative diagram of maps between Kuranishi spaces
\begin{equation}\label{localizzazione-def}
    \xymatrix{\Def(\widehat{X})\ar[r]^-{\widehat{l}_p}
    \ar@{^{(}->}[d]^-{D}&\Def(\widehat{U}_p)\ar@{^{(}->}[d]^-{D_{loc}}\\
    \Def(X)\ar[r]^-{l_p}&\Def(U_p)\cong T^1_{U_p}}
\end{equation}
where the horizontal maps are the natural localization maps while
the vertical maps are \emph{ injective maps} induced by the
resolution $\phi$ (see \cite{Wahl76} Propositions 1.8 and 1.12,
\cite{Kollar-Mori92} Proposition (11.4)). Then, by explicit calculations (see \cite{Rnami1}, Thm. 4), it turns out that
\begin{equation}\label{nami-vanishing}
    \left.\begin{array}{c}
      \dim \Def(\widehat{U}_p)=1 \\
      \im(l_p)\cap\im(D_{loc})= 0 \\
    \end{array}\right\}\ \Longrightarrow\ \im(\widehat{l}_p) = 0
\end{equation}
meaning that
\begin{itemize}
    \item[(5)] \emph{no global deformation of $\widehat{X}$ may induce a local non--trivial
    deformation of $\widehat{U}_p$; in particular $\phi^{-1}(p)$ turns out to be a \emph{rigid}
    $A_2$ exceptional tree and $T(\widehat{X},X,\widetilde{X})$ cannot be def-equivalent to a conifold
    transition}.
\end{itemize}}
\end{example}

\subsection{A necessary condition for simplicity of small transitions}

The previous example allows us to understand some further necessary condition that a small g.t. should satisfy to be a simple g.t.:

\begin{proposition}\label{prop:small&simple} Recall the definition of $\Theta_{\bullet}$ as the `tangent'' sheaf. Then if $T(Y,\overline{Y},\widetilde{Y})$ is a simple small geometric transition of \cy 3--folds then
\begin{equation}\label{cohom-condizione}
    h^1(\overline{Y},\Theta_{\overline{Y}})<h^1(Y,\Theta_Y)\ .
\end{equation}
\end{proposition}

\begin{proof} The proof is an application of R.~Friedman techniques presented in \cite{Friedman86}. In fact by the Leray spectral sequence applied to the birational small contraction $\phi:Y\rightarrow\overline{Y}$ and the local to global spectral sequence relating the global $\T^{\bullet}_{\overline{Y}}:=\Ext^{\bullet}(\Omega_{\overline{Y}},\mathcal{O}_{\overline{Y}})$ with the sheaves $\Theta^{\bullet}_{\overline{Y}}:=\ext^{\bullet}(\Omega_{\overline{Y}},\mathcal{O}_{\overline{Y}})$, one gets the following commutative diagram
\begin{equation}\label{Friedman-diagramma}
    \xymatrix{0\rightarrow H^1(Y,R^0\phi_*\Theta_Y)\ar[r]\ar@<10pt>@{=}[d] &
    \mathbb{T}^1_Y\ar[r]^-{\lambda}\ar[d]^{\delta} &
    H^0\left(\overline{Y},R^1\phi_*\Theta_Y\right)\ar[r]\ar[d]^{\delta_{loc}}& \cdots\\
    0\rightarrow H^1(\overline{Y},\Theta_{\overline{Y}})\ar[r]&\mathbb{T}^1_{\overline{Y}}\ar[r]^-{\overline{\lambda}}
     &T^1_{\overline{Y}}:=H^0(\overline{Y},\Theta^1_{\overline{Y}})\ar[r]& \cdots}
\end{equation}
where
\begin{itemize}
  \item the vertical equality comes from an application of Hartogs Theorem giving $R^0\phi_*\Theta_{Y}\cong \Theta_{\overline{Y}}$ (\cite{Friedman86} Lemma (3.1)),
  \item the vertical morphism $\delta$ is the differential of an injective map between Kuranishi spaces $\Def(Y)\hookrightarrow\Def(\overline{Y})$, constructed by J.M.~Wahl \cite[$\S1$]{Wahl76} (see also \cite[Prop.~(2.1)]{Friedman86} and \cite[Prop.~(11.4)]{Kollar-Mori92}), which turns out to be still injective since $\delta_{loc}$ is injective,
  \item $\delta_{loc}$ is the localization of $\delta$ near to $\Sing(\overline{Y})$, which is injective by a result of Friedman (\cite{Friedman86}, Prop. (2.1)).
\end{itemize}
By (\ref{T-lisci}), $\T^1_Y\cong H^1(Y,\Theta_Y)$ and the fact that $T$ is a simple g.t. guarantees the existence of a global deformation of $Y$ inducing a first order deformation $\xi\in H^1(Y,\Theta_Y)$ of $Y$ such that $\xi_{loc}:=\delta_{loc}\circ\lambda(\xi)$ gives non trivial first order deformations of \emph{any} singularity $p\in\Sing(\overline{Y})$ to ordinary double points. Since $\delta_{loc}$ is injective, this means that $\im\lambda\neq 0$ and the exactness of the upper sequence in (\ref{Friedman-diagramma}) gives necessarily the cohomological condition (\ref{cohom-condizione}).
\end{proof}

\begin{remark} {\rm Back to the Namikawa's example \ref{ex:nami}, let us observe that for the g.t. $T(\widehat{X},X,\widetilde{X})$, where $X=S\times_{\P^1}S$ is the fibred self-product of a cuspidal elliptic surface, one gets $h^1(\overline{Y},\Theta_{\overline{Y}})=h^1(Y,\Theta_Y)=3$.}
\end{remark}

\subsection{An example of a simple small geometric transition}\label{ex:small&simple}

Let us consider the singular quintic threefold $\overline{Q}\subset\P^4$ given by
\begin{equation}\label{quintica}
    u(u-2x)(u-3y)(x^2-y^2)-(z^5-w^5)=0\ .
\end{equation}
The singular locus $\Sing(\overline{Q})$ is composed by 10 isolated hypersurface singularities, each of them analytically equivalent to the one described by the local equation
\begin{equation}\label{local-eq}
    x^2-y^2=z^5-w^5
\end{equation}
which is a $cA_4$ singular point whose Milnor and Tyurina numbers are equal to 16. A resolution of this singular point is obtained by a \emph{successive blow up} of the planes
$$\pi_i:x-y=z-\epsilon^iw=0\quad,\quad 0\leq i\leq 3\ ,\ \epsilon^5=1\ .$$
More precisely: blow up $\C^4$ along $\pi_0$, then blow up again along the strict transform of $\pi_1$ and so on. At the end look at the strict transform of the singularity, which carries an exceptional locus composed by a tree of 4 lines dually represented by the Dynkin graph $A_4$.

\noindent We are now in a position to construct a non-conifold geometric transition as follows:
\begin{itemize}
  \item \emph{the resolution:} the quintic threefold $\overline{Q}$ admits a global resolution $\widehat{Q}$ which can be obtained by the successive blow up of 16 planes
      $$\pi_i^j:l_j=z-\epsilon^i w=0\quad 0\leq i\leq 3\ ,\ 1\leq j\leq 4$$
      where $\{l_1,\ldots,l_4\}\subset\{u,u-2x,u-3y,x-y,x+y\}$.
  \item \emph{the smoothing:} it is obviously given by a smooth quintic threefold $Q\subset\P^4$.
\end{itemize}
This gives the g.t. $T(\widehat{Q},\overline{Q},Q)$.
To deform $T$ to a conifold transition consider the following deformation $\overline{Q}_{(a,b,c)}$ of $\overline{Q}$
$$
 u(u-2x)(u-3y)(x^2-y^2)-(z-w)(z-\epsilon w)(z-\epsilon^2w+a)(z-\epsilon^3w+b)(z-\epsilon^4w+c)=0
$$
which, for a general $\alpha:=(a,b,c)\in\C^3$, splits up each singular point of $\overline{Q}$ into 10 nodes, hence giving 100 nodes. Since the deformation $\overline{Q}_{\alpha}$ respects the factorization in the equation of $\overline{Q}$, it lifts to a deformation $\widehat{Q}_{\alpha}$ of the resolution $\widehat{Q}$ splitting up every exceptional $A_4$ tree into 10 disjoint lines. This gives a deformation family of morphisms
\begin{equation*}
    \xymatrix{\widehat{\mathcal{Q}}\ar[rr]^-{\Phi}\ar[dr]_-f&&\overline{\mathcal{Q}}\ar[dl]^-g\\
    &\C^3&}
\end{equation*}
hence a def-equivalence $T\sim T_{\alpha}(\widehat{Q}_{\alpha},\overline{Q}_{\alpha},Q)$. Let us further observe that the deformations $\overline{Q}_{\alpha}$ are not all distinct up to isomorphisms: if we consider the Kuranishi space $T^1$ of any singularity of $\overline{Q}$, there is a well defined map $\C^3\rightarrow T^1$ whose image is 1-dimensional. This is enough to show that
$$\dim\left(\im\left(\lambda:\T^1_{\widehat{Q}}\rightarrow H^0(\overline{Q},R^1\phi_*\Theta_{\widehat{Q}}\right)\right)=1$$
giving $h^1(\Theta_{\overline{Q}})=17<18=h^1(\Theta_{\widehat{Q}})$, which is consistent with Proposition \ref{prop:small&simple}.
The further main invariants of the g.t. $T$ and the conifold t. $T_{\alpha}$ are listed in the following table:
\begin{equation*}
\begin{tabular}{lclclclclclclcl}
  \hline \hline
  Variety         &\quad     & $h^1(\Theta_{\bullet})$&\quad  & $b_2$&\quad
  &$\rho$&\quad  & $b_3$ &\quad & $b_4$&\quad  &$\chi$    \\
  \hline
    &&&&&&&\\
  $\widehat{Q},\widehat{Q}_{\alpha}$    &\quad    & 18 &\quad & 17 &\quad & 17  &\quad & 36 &\quad & 17 &\quad & 0 \\
  \hline
  &&&&&&&\\
  $\overline{Q}$            &\quad      &  17 &\quad & 1&\quad  & 1 &\quad & 60&\quad  & 17 &\quad & -40 \\
  \hline
  &&&&&&&\\
  $\overline{Q}_{\alpha}$            &\quad      &  18 &\quad & 1&\quad  & 1 &\quad & 120&\quad  & 17 &\quad & -100 \\
  \hline
  &&&&&&&\\
  $Q$               &\quad   &  101 &\quad & 1 &\quad & 1&\quad  & 204&\quad  & 1&\quad  & -200  \\
   \hline\hline
\end{tabular}
\end{equation*}
They can be computed from the well known invariants of the smooth quintic threefold $Q$ by means of relations given in \cite{R1}, Thm. 7.

\section*{Acknowledgments}

A first draft of this paper was written on the occasion of the GTM Seminar held in Genova on March 21-22, 2013, and dedicated to Prof.~D.~Gallarati for his 90th birthday.

First of all I would like to thank the organizers A.~Alzati, G.~Casnati, F.~Galluzzi, R.~Notari, M.E.~Rossi, G.~Valla and, last but not least, L.~Badescu and  M.~Bel\-tra\-met\-ti, for the warm welcome and hospitality and all the participants for the pleasant climate they were able to establish during the two days of work. Among them a special thanks to F.~Catanese for stimulating conversations and suggestions.

Finally I would like to thank R.~Davies for his helpful comments, C.~Casa\-gran\-de for having pointed out an inaccuracy in a previous version of Remark \ref{rem:fisica} and unknown referees for their useful suggestions and improvements.

\end{document}